\documentclass[oneside,english]{amsart}
\usepackage[T1]{fontenc}
\usepackage[latin9]{inputenc}
\usepackage{amstext}
\usepackage{amsthm}
\usepackage{amssymb}
\usepackage{comment}

\makeatletter
\numberwithin{equation}{section}
\numberwithin{figure}{section}
\theoremstyle{plain}
\newtheorem{thm}{\protect\theoremname}
\theoremstyle{remark}
\newtheorem{rem}[thm]{\protect\remarkname}

\usepackage{amsrefs}

\newtheorem{hyp}[thm]{Hypothesis}

\makeatother

\usepackage{babel}
\providecommand{\remarkname}{Remark}
\providecommand{\theoremname}{Theorem}

\begin{document}
\title{Principal series component of gelfand-graev representation}
\author{Manish Mishra}
\email{manish@iiserpune.ac.in}
\thanks{The first named author was partially supported by SERB MATRICS 
and SERB ERCA grants}
\author{Basudev Pattanayak}
\email{basudev.pattanayak@students.iiserpune.ac.in}
\thanks{The second named author was supported by CSIR fellowship, Govt. of India}
\address{Indian Institute of Science Education and Research Pune, Dr. Homi
Bhabha Road, Pasha, Pune 411008, Maharashtra, India}

\begin{abstract}
Let $G$ be a connected reductive group defined over a non-archimedean
local field $F$. Let $B$ be a minimal $F$-parabolic subgroup with
Levi factor $T$ and unipotent radical $U$. Let $\psi$ be a non-degenerate
character of $U(F)$ and $\lambda$ a character of $T(F)$. Let $(K,\rho)$
be a Bushnell-Kutzko type associated to the Bernstein block of $G(F)$
determined by the pair $(T,\lambda)$. We study the $\rho$-isotypical
component $(c\text{-ind}_{U(F)}^{G(F)}\psi)^{\rho}$ of the induced
space $c\text{-ind}_{U(F)}^{G(F)}\psi$ of functions compactly supported
mod $U(F)$. We show that $(c\text{-ind}_{U(F)}^{G(F)}\psi)^{\rho}$
is cyclic module for the Hecke algebra $\mathcal{H}(G,\rho)$ associated
to the pair $(K,\rho)$. When $T$ is split, we describe it more explicitly
in terms of $\mathcal{H}(G,\rho)$. We make assumptions on the residue
characteristic of $F$ and later also on the characteristic of $F$
and the center of $G$ depending on the pair $(T,\lambda)$. Our results
generalize the main result of Chan and Savin in \cite{CS18} who treated
the case of $\lambda=1$ for $T$ split. 
\end{abstract}

\maketitle

\section{Introduction}

Let $G$ be a connected reductive group defined over a non-archimedean
local field $F$. Fix a minimal $F$-parabolic subgroup $B=TU$ of
$G$ with unipotent radical $U$ and whose Levi factor $T$ contains
a maximal $F$-split torus of $G$. Let $\psi$ be a nondegenerate
character of $U(F)$ and consider the induced representation $c\text{-ind}_{U(F)}^{G(F)}\psi$
realized by functions whose support is compact mod $U(F)$. It is
a result of Bushnell and Henniart \cite{BH03} that the Bernstein
components of $c\text{-ind}_{U(F)}^{G(F)}\psi$ are finitely generated. 

Now let $\lambda$ be a character of $T(F)$. The pair $(T,\lambda)$
determines a Bernstein block $\mathcal{R}^{[T,\lambda]_{G}}(G(F))$
in the category of smooth representations $\mathcal{R}(G(F))$ of
$G(F)$. Bushnell-Kutzko types are known to exist for Bernstein blocks
under suitable residue characteristic hypothesis \cite{KY17, Fin}.
Let $(K,\rho)$ be a $[T,\lambda]_{G}$-type and let $\mathcal{H}(G,\rho)$
be the associated Hecke algebra. We show in Theorem \ref{thm:cyc}
that the $\rho$-isotypical component $(c\text{-ind}_{U(F)}^{G(F)}\psi)^{\rho}$
is a cyclic $\mathcal{H}(G,\rho)$-module. 

Now assume that $T$ is split and $\psi$ is a non-degenerate character of $U(F)$ of generic depth zero (see \S \ref{sec:statement}). If $\lambda\neq1$, then assume further
that the center of $G$ is connected. In that case, $\mathcal{H}(G,\rho)$
is an Iwahori-Hecke algebra. It contains a finite subalgebra $\mathcal{H}_{W,\lambda}$.
The algebra $\mathcal{H}_{W,\lambda}$ has a one dimensional representation
$\text{sgn}$. We show in Theorem \ref{thm:main} that the $\mathcal{H}(G,\rho)$-module
$(c\text{-ind}_{U(F)}^{G(F)}\psi)^{\rho}$ is isomorphic to $\mathcal{H}(G,\rho)\underset{\mathcal{H}_{W,\lambda}}{\otimes}\text{sgn}$.
For positive depth $\lambda$, Theorem \ref{thm:main} assumes that
$F$ has characteristic $0$ and its residue characteristic is not
too small. 

Theorems \ref{thm:cyc} and \ref{thm:main} generalize the main result
of Chan and Savin in \cite{CS18} who treat the case $\lambda=1$
for $T$ split, i.e., unramified principal series blocks of split
groups. Our proofs benefit from the ideas in \cite{CS18}. However
they are quite different. The existence of a generator in $(c\text{-ind}_{U(F)}^{G(F)}\psi)^{\rho}$
is concluded by specializing quite general results in \cite{BH03, BK98}.
For Theorem \ref{thm:main}, instead of computing the effect of intertwiners
on the generator as in \cite{CS18}, we make a reduction to depth-zero
case and then to a finite group analogue of the question. There it
holds by a result of Reeder \cite{Re02}*{\S 7.2}.

\section{Notations}

Throughout, $F$ denotes a non-archimedean local field. Let $\mathcal{O}$
denote the ring of integers of $F$. We denote by $q$, the cardinality
of the residue field $\mathbb{F}_{q}$ of $F$ and by $p$ the characteristic
of $\mathbb{F}_{q}$. 

\section{\label{sec:roche}Preliminaries}

We use this section to recall some basic theory and also fix notation. 

\subsection{Bernstein decomposition}

Let $\mathcal{R}(G(F))$ denote the category of smooth complex representations
of $G(F)$. The \textit{Bernstein decomposition} gives a direct product
decomposition of $\mathcal{R}(G(F))$ into indecomposable subcategories:
\[
\mathcal{R}(G(F))=\prod_{\mathfrak{s}\in\mathfrak{B}(G)}\mathcal{R}^{\mathfrak{s}}(G(F)).
\]
 Here $\mathfrak{B}(G)$ is the set of \textit{inertial equivalence
classes, }i.e., equivalence classes $[L,\sigma]_{G}$ of cuspidal
pairs $(L,\sigma)$, where $L$ is an $F$-Levi, $\sigma$ is a supercuspidal
of $L(F)$ and where the equivalence is given by conjugation by $G(F)$
and twisting by unramified characters of $L(F)$. The block $\mathcal{R}^{[L,\sigma]_{G}}(G(F))$
consists of those representations $\pi$ for which each irreducible
constituent of $\pi$ appears in the parabolic induction of some supercuspidal
representation in the equivalence class $[L,\sigma]_{G}$. 

\subsection{Hecke algebra \cite{BK98}}

Let $(\tau,V)$ be an irreducible representation of a compact open
subgroup $J$ of $G(F)$. The Hecke algebra $\mathcal{H}(G,\tau)$
is the space of compactly supported functions $f:G(F)\rightarrow\mathrm{End}_{\mathbb{C}}(V^{\vee})$
satisfying,
\[
f(j_{1}gj_{2})=\tau^{\vee}(j_{1})f(g)\tau^{\vee}(j_{2})\text{ }\text{for all }j_{1,}j_{2}\in J\text{ }\text{and }g\in G(F).
\]
Here $(\tau^{\vee},V^{\vee})$ denotes the dual of $\tau$. The standard
convolution operation gives $\mathcal{H}(G,\tau)$ the structure of
an associative $\mathbb{C}$-algebra with identity. 

Let $\mathcal{R}_{\tau}(G(F))$ denote the subcategory of $\mathcal{R}(G(F))$
whose objects are the representations $(\pi,\mathcal{V})$ of $G(F)$
generated by the $\tau$-isotypic subspace $\mathcal{V}^{\tau}$ of
$\mathcal{V}$. There is a functor 
\[
\mathbf{M}_{\tau}:\mathcal{R}_{\tau}(G(F))\rightarrow\mathcal{H}(G,\tau)\text{-Mod},
\]
 given by 
\[
\pi\mapsto\mathrm{Hom}_{J}(\tau,\pi).
\]
 Here $\mathcal{H}(G,\tau)\text{-Mod}$ denotes the category of unital
left modules over $\mathcal{H}(G,\tau)$. 

For $\mathfrak{s}\in\mathfrak{B}(G)$, the pair $(J,\tau)$ is an
$\mathfrak{s}$-type if $\mathcal{R}_{\tau}(G(F))=\mathcal{R}^{\mathfrak{s}}(G(F))$.
In that case, the functor $\mathbf{M}_{\tau}$ gives an equivalence
of categories. 

\subsection{G-cover \cite{BK98}}

Let $(K_{M,}\rho_{M})$ be a $[M,\sigma]_{M}$-type. Let $(K,\rho)$
be a pair consisting of a compact open subgroup $K$ or $G(F)$ and
an irreducible representation $\rho$ of $K$. Suppose that for any
opposite pair of $F$-parabolic subgroups $P=MN$ and $\bar{P}=M\bar{N}$
with Levi factor $M$ and unipotent radicals $N$ and $\bar{N}$ respectively,
the pair $(K,\rho)$ satisfies the following properties:

\begin{enumerate}

\item[(1)]$K$ decomposes with respect to $(N,M,\bar{N})$, i.e.,
\[
K=(K\cap N)\ldotp(K\cap M)\ldotp(K\cap\bar{N}).
\]

\item[(2)]$\rho|K_{M}=\rho_{M}$ and $K\cap N$, $K\cap\bar{N}\subset\mathrm{ker}(\rho)$. 

\item[(3)]For any smooth representation $(\pi,V)$ of $G(F),$ the
natural projection $V$ to the Jacquet module $V_{N}$ induces an
injection on $V^{\rho}$. 

\end{enumerate}

The pair $(K,\rho)$ is then called the $G$-cover of $(K_{M},\rho_{M})$.
See \cite{Bl}*{Theorem 1} for this reformulation of the original
definition of $G$-cover due to Bushnell and Kutzko \cite{BK98}*{\S 8.1}
(see also \cite{KY17}*{\S 4.2}). If $(K,\rho)$ is a $G$-cover of
$(K_{M},\rho_{M})$, then it is an $[M,\sigma]_{G}$-type. 

Suppose $(K,\rho)$ is a $G$-cover of $(K_{M},\rho_{M})$, then for
any $F$-parabolic subgroup $P'=MN'$ with Levi factor $M$ and unipotent
radical $N'$, there is a $\mathbb{C}$-algebra embedding \cite{BK98}*{\S 8.3}
\begin{equation}
t_{P'}:\mathcal{H}(M,\rho_{M})\rightarrow\mathcal{H}(G,\rho),\label{eq:HMG}
\end{equation}
with the property that for any smooth representation $\Upsilon$ of
$G(F)$,
\begin{equation}
\mathbf{M}_{\rho_{M}}(\Upsilon_{N'})\cong t_{P'}^{*}(\mathbf{M}_{\rho}(\Upsilon)).\label{eq:VMG}
\end{equation}
 Here $t_{P'}^{*}:\mathcal{H}(G,\rho)\text{-mod }\rightarrow\mathcal{H}(M,\rho_{M})\text{-mod }$
induced by $t_{P'}$. 

Kim and Yu \cite{KY17}, using Kim's work \cite{Kim07}, showed that
Yu's construction of supercuspidals \cite{Yu01} can be used to produce
$G$-covers of $[M,\sigma]_{M}$-types for all $[M,\sigma]_{G}\in\mathfrak{B}(G)$,
assuming $F$ has characteristic $0$ and the residue characteristic
$p$ of $F$ is suitably large. Recently Fintzen \cite{Fin}, using
an approach different from Kim, has shown the construction of types
for all Bernstein blocks without any restriction on the characteristic
of $F$ and assuming only that $p$ does not divide the order of the
Weyl group of $G$. 

\section{\label{sec:GG}Gelfand-Graev spaces}

Let $G$ be a connected reductive group defined over $F$. Fix a maximal
$F$-split torus $S$ in $G$ and let $T$ denote its centralizer.
Then $T$ is the Levi factor of a minimal $F$-parabolic subgroup
$B$ of $G$. We denote the unipotent radical of $B$ by $U$. A smooth
character 
\[
\psi:U(F)\rightarrow\mathbb{C}^{\times}
\]
 is called non-degenerate if its stabilizer in $S(F)$ lies in the
center $Z$ of $G$. 

The Gelfand-Graev representation $c\text{-ind}_{U(F)}^{G(F)}\psi$
of $G(F)$ is provided by the space of right $G(F)$-smooth compactly
supported modulo $U(F)$ functions $f:G(F)\rightarrow\mathbb{C}$
satisfying:
\[
f(ug)=\psi(u)f(g),\text{ }\forall u\in U(F),g\in G(F).
\]

Let $M$ be a $(B,T)$-standard $F$-Levi subgroup of an $F$-parabolic
$P=MN$ of $G$, i.e., $M$ contains $T$ and $P$ contains $B$.
Then $B\cap M$ is a minimal parabolic subgroup of $M$ with unipotent
radical $U_{M}:=U\cap M$. Also, $\psi_{M}:=\psi|M$ is a non-degenerate
character of $U_{M}(F)$. As before, denote by $\bar{P}=M\bar{N}$,
the opposite parabolic subgroup. We also have an isomorphism of $M(F)$
representations \cite{BH03}*{\S 2.2, Theorem}
\begin{equation}
c\text{-ind}_{U_{M}(F)}^{M(F)}\psi_{M}\cong(c\text{-ind}_{U(F)}^{G(F)}\psi)_{\bar{N}}.\label{eq:cindU}
\end{equation}

Now let $\sigma$ be a character of $T(F)$. Let $(K_{T},\rho_{T})$
be a $[T,\sigma]_{T}$-type and let $(K,\rho)$ denote its $G$-cover.
We assume that the residue characteristic $p$ does not divide the
order of the Weyl group of $G$, so that $(K,\rho)$ exists by \cite{Fin}.
Let $\bar{B}=T\bar{U}$ denote the opposite Borel. View $\mathcal{H}(T,\rho_{T})$
as a subalgebra of $\mathcal{H}(G,\rho)$ via the embedding: 
\[
t_{\bar{B}}:\mathcal{H}(T,\rho_{T})\rightarrow\mathcal{H}(G,\rho),
\]
of Equation (\ref{eq:HMG}). 
\begin{thm}
\label{thm:cyc}There is an isomorphism 
\[
(c\emph{-ind}_{U(F)}^{G(F)}\psi)^{\rho}\cong\mathcal{H}(T,\rho_{T})
\]
of $\mathcal{H}(T,\rho_{T})$-modules. Consequently, $(c\emph{-ind}_{U(F)}^{G(F)}\psi)^{\rho}$
is a cyclic $\mathcal{H}(G,\rho)$-module. 
\end{thm}

\begin{proof}
Putting $M=T$ in Equation (\ref{eq:cindU}) and observing that in
this case, $U_{M}=1$, we get an isomorphism of $T(F)$ representations
\begin{eqnarray*}
(c\text{-ind}_{U(F)}^{G(F)}\psi)_{\bar{U}} & \cong & c\text{-ind}_{1}^{T(F)}\mathbb{C}\\
 & \cong & C_{c}^{\infty}(T(F)).
\end{eqnarray*}
Consequently, 
\begin{eqnarray*}
(c\text{-ind}_{U(F)}^{G(F)}\psi)_{\bar{U}}^{\rho_{T}} & \cong & C_{c}^{\infty}(T(F))^{\rho_{T}}\\
 & \cong & \mathcal{H}(T,\rho_{T})
\end{eqnarray*}
 as $\mathcal{H}(T,\rho_{T})$-modules. Now by Equation (\ref{eq:VMG}),
\[
(c\text{-ind}_{U(F)}^{G(F)}\psi)_{\bar{U}}^{\rho_{T}}\cong(c\text{-ind}_{U(F)}^{G(F)}\psi)^{\rho}
\]
 as $\mathcal{H}(T,\rho_{T})$-modules. The result follows. 
\end{proof}

\section{\label{sec:PSC}Principal series component}

\subsection{Some results of Roche }

In this subsection, we summarize some results of Roche in \cite{Ro98}. 

Let the notations be as in Section \ref{sec:GG}. We assume further
that $S=T$, so that $B=TU$ is now an $F$-Borel subgroup of $G$
containing the maximal $F$-split torus $T$. The pair $(B,T)$ determines
a based root datum $\Psi=(X,\Phi,\Pi,X^{\vee},\Phi^{\vee},\Pi^{\vee})$.
Here $X$ (resp. $X^{\vee}$) is the character (resp. co-character)
lattice of $T$ and $\Pi$ (resp. $\Pi^{\vee}$) is a basis (resp.
dual basis) for the set of roots $\Phi=\Phi(G,T)$ (resp. $\Phi^{\vee}$)
of $T$ in $G$. 

For the results of this subsection, we assume that $F$ has characteristic
$0$ and the residue characteristic $p$ of $F$ satisfies the following
hypothesis.

\begin{hyp}\label{hyp}

If $\Phi$ is irreducible, $p$ is restricted as follows:
\begin{itemize}
\item for type $A_{n}$, $p>n+1$
\item for type $B_{n},C_{n},D_{n}$ $p\neq2$
\item for type $F_{4}$, $p\neq2,3$
\item for types $G_{2},E_{6}$ $p\neq2,3,5$
\item for types $E_{7},E_{8}$ $p\neq2,3,5,7$ 
\end{itemize}
If $\Phi$ is not irreducible, then $p$ excludes primes attached
to each of its irreducible factors. 

\end{hyp}

We let $\mathcal{T}=T(\mathcal{O})$ denote the maximal compact subgroup
of $T(F)$, $N_{G}(T)$ to be the normalizer of $T$ in $G$ and $W=W(G,T)=N_{G}(T)/T=N_{G}(T)(F)/T(F)$
denote the Weyl group of $G$. 

Let $\chi^{\#}$ be a character of $T(F)$ and put $\chi=\chi^{\#}|T(F)_{0}$,
where $T(F)_{0}$ denotes the maximal compact subgroup of $T(F)$.
Then $(T(F)_{0},\chi)$ is a $[T,\chi^{\#}]_{T}$-type.

Let $N_{G}(T)(F)_{\chi}$ (resp. $N_{G}(T)(\mathcal{O})_{\chi}$,
resp. $W_{\chi}$) denote the subgroup of $N_{G}(T)(F)$ (resp. $N_{G}(T)(\mathcal{O})$,
resp. $W$) which fixes $\chi$. The group $N_{G}(T)(F)_{\chi}$ contains
$T(F)$ and we have $W_{\chi}=N_{G}(T)(F)_{\chi}/T(F)$. Denote by
$\mathcal{W}=\mathcal{W}(G,T)=N_{G}(T)(F)/\text{\ensuremath{\mathcal{T}} }$,
the Iwahori-Weyl group of $G$. There is an identification $N_{G}(T)(F)=X^{\vee}\rtimes N_{G}(T)(\mathcal{O})$
given by the choice of a uniformizer of $F$. Since $N_{G}(T)(\mathcal{O})/\mathcal{T}=N_{G}(T)/T$,
this identification also gives an identification $\mathcal{W}=X^{\vee}\rtimes W$.
Let $\mathcal{W}_{\chi}=X^{\vee}\rtimes W_{\chi}$ be the subgroup
of $\mathcal{W}$ which fixes $\chi$. 

Let 
\begin{eqnarray*}
\Phi^{\prime} & = & \{\alpha\in\Phi\mid\chi\circ\alpha^{\vee}|_{\mathcal{O}^{\times}}=1\}.
\end{eqnarray*}
Then $\Phi^{\prime}$ is a sub-root system of $\Phi$. Let $s_{\alpha}$
denote the reflection on the space $\mathcal{A}=X^{\vee}\otimes_{\mathbb{Z}}\mathbb{R}$
associated to a root $\alpha\in\Phi$ and write $W^{\prime}=\langle s_{\alpha}\mid\alpha\in\Phi^{\prime}\rangle$
to be the associated Weyl group. Let $\Phi^{+}$ (resp. $\Phi^{-}$)
be the system of positive (resp. negative) roots determined by the
choice of the Borel $B$ and let $\Phi^{\prime+}=\Phi^{+}\cap\Phi^{\prime}$.
Then $\Phi^{\prime+}$ is a positive system in $\Phi^{\prime}$. Put
\[
C_{\chi}=\{w\in W_{\chi}\mid w(\Phi^{\prime+})=\Phi^{\prime+}\}.
\]
 Then we have,
\[
W_{\chi}=W^{\prime}\rtimes C_{\chi}.
\]

The character $\chi$ extends to a $W_{\chi}$-invariant character
$\tilde{\chi}$ of $N_{G}(T)(\mathcal{O})_{\chi}$. Denote by $\tilde{\chi}$,
the character of $N_{G}(T)(F)_{\chi}$ extending $\tilde{\chi}$ trivially
on $X^{\vee}$. 

Roche's construction produces a $[T,\chi^{\#}]_{G}$-type $(K,\rho)$.
The pair $(K,\rho)$ depends on the choice of $B,T,\chi$ but not
on the extension $\chi^{\#}$ of $\chi$. Denote by $I_{G(F)}(\rho)$,
the set of elements in $G(F)$ which intertwine $\rho$. Equivalently,
$g\in I_{G(F)}(\rho)$ iff the double coset $KgK$ supports a non-zero
function in $\mathcal{H}(G,\rho)$. We have an equality 
\begin{equation}
I_{G(F)}(\rho)=K\mathcal{W_{\chi}}K.\label{eq:IGrho}
\end{equation}

For an element $w\in\mathcal{W}_{\chi}$, choose any representative
$n_{w}$ of $w$ in $N_{G}(T)(F)_{\chi}$ and let $f_{\tilde{\chi},w}$
be the unique element of the Hecke algebra $\mathcal{H}(G,\rho)$
supported on $Kn_{w}K$ and having value $q^{-\ell(w)/2}\tilde{\chi}(n_{w})^{-1}$.
Here $\ell$ is the length function on the affine Weyl group $\mathcal{W}$.
The functions $f_{\tilde{\chi},w}$ for $w\in\mathcal{W}_{\chi}$
form a basis for the $\mathbb{C}$-vector space $\mathcal{H}(G,\rho)$.
Denote by $\mathcal{H}_{W,\chi}$ the subalgebra of $\mathcal{H}(G,\rho)$
generated by $\{f_{\tilde{\chi},w}\mid w\in W'\}$. Also, identify
$\mathcal{H}(T,\chi)$ as a subalgebra of $\mathcal{H}(G,\rho)$ using
the embedding $t_{B}$. When $G$ has connected center, $C_{\chi}=1$ assuming 
Hypothesis \ref{hyp}. In that case, $\mathcal{H}_{W,\chi}$ 
and $\mathcal{H}(T,\chi)$ together
generate the full Hecke algebra $\mathcal{H}(G,\rho)$. 

\subsection{Statement of Theorem}\label{sec:statement}
We continue to assume that $G$ is split. 
Extend the triple $(G,B,T)$ to a Chevalley-Steinberg pinning of $G$.
This determines a hyperspecial point $x$ in the Bruhat-Tits building
which gives $G$ the structure of a Chevalley group. With this identification,
$(G,B,T)$ are defined over $\mathcal{O}$ and the hyperspecial subgroup
$G(F)_{x,0}$ at $x$ is $G(\mathcal{O})$. Let $G(F)_{x,0+}$ denote the pro-unipotent radical of $G(F)_{x,0}$. Then $G(F)_{x,0}/G(F)_{x,0+}\cong G(\mathbb{F}_{q})$. We say that a non-degenerate character $\psi$ 
of $U(F)$ is of \emph{generic depth zero} at $x$ if $\psi|U(F)\cap G(F)_{x,0}$ factors through a
generic character of $\text{\ensuremath{\underbar{\ensuremath{\psi}}}}$
of $U(\mathbb{F}_{q})=U(F)\cap G(F)_{x,0}/U(F)\cap G(F)_{x,0+}$ (see \cite{DeRe10}*{\S 1} for 
a more general definition). Note that if $G$ has connected center, then all non-degenerate
characters of $U(F)$ form a single orbit under $T(F)$.

Let $\mathrm{sgn}$ denote the one dimensional representation of $\mathcal{H}_{W,\chi}$
in which $f_{\tilde{\chi},w}$ acts by the scalar $(-1)^{\ell'(w)}$.
Here $\ell'$ denotes the length function on $W'$. 
\begin{thm}
\label{thm:main}
Let $\psi$ be a non-degenerate character of $U(F)$ of generic depth zero at $x$. If $\chi\neq1$, then assume that the center of $G$ is connected. If $\chi$ has positive depth, then assume further that
$F$ has characteristic $0$ and the residue characteristic satisfies
Hypothesis \ref{hyp}. Then $\mathcal{H}(G,\rho)$-module $(c\emph{-ind}_{U(F)}^{G(F)}\psi)^{\rho}$
is isomorphic to $\mathcal{H}(G,\rho)\underset{\mathcal{H}_{W,\chi}}{\otimes}\mathrm{sgn}$. 
\end{thm}

\section{Proof of Theorem \ref{thm:main}}

We retain the notations introduced in Sections \ref{sec:GG} and \ref{sec:PSC}.

\subsection{Reduction to depth-zero}

It follows from the the proof of \cite{Ro98}*{Theorem 4.15} (see
also loc. cit., page 385, 2nd last paragraph), that there exists a
standard $F$-Levi subgroup $M$ of $G$ which is the Levi factor
of a standard parabolic $P=MN$ of $G$ and which is minimal with
the property that 
\begin{equation}
I_{G(F)}(\rho)\subset KM(F)K.\label{eq:IGrho_sub}
\end{equation}
Put $(K_{M,}\rho_{M})=(K\cap M,\rho|M(F))$. From \cite{BK98}*{Theorem 7.2(ii)},
it follows that $(K,\rho)$ satisfies the requirements \cite{BK98}*{\S 8.1}
of being $G$-cover of $(K_{M},\rho_{M})$. It also follows from \cite{BK98}*{Theorem 7.2(ii)}
that there is a support preserving Hecke algebra isomorphism
\begin{equation}
\Psi^{M}:\mathcal{H}(M,\rho_{M})\overset{\simeq}{\rightarrow}\mathcal{H}(G,\rho).\label{eq:H-isom}
\end{equation}

By Equation (\ref{eq:cindU}), we have an isomorphism of $\mathcal{H}(M,\rho_{M})$-modules
\begin{equation}
(c\text{-ind}_{U_{M}(F)}^{M(F)}\psi_{M})^{\rho_{M}}\cong((c\text{-ind}_{U(F)}^{G(F)}\psi)_{\bar{N}(F)})^{\rho_{M}}.\label{eq:c-ind1}
\end{equation}

And by Equation (\ref{eq:VMG}), we have a $\Psi^{M}$-equivariant
isomorphism 
\begin{equation}
(c\text{-ind}_{U(F)}^{G(F)}\psi)_{\bar{N}(F)})^{\rho_{M}}\cong(c\text{-ind}_{U(F)}^{G(F)}\psi)^{\rho}.\label{eq:c-ind2}
\end{equation}
Combining Equations (\ref{eq:c-ind1}) and (\ref{eq:c-ind2}), we
get a $\Psi^{M}$-equivariant isomorphism 
\[
(c\text{-ind}_{U_{M}(F)}^{M(F)}\psi_{M})^{\rho_{M}}\cong(c\text{-ind}_{U(F)}^{G(F)}\psi)^{\rho}.
\]
Also it is shown in the proof of \cite{Ro98}*{Theorem 4.15} that
for such an $M$, there is a character $\chi_{1}$ of $M(F)$ such
that $\chi\chi_{1}$ viewed as a character of $T(F)_{0}$ is depth-zero.
We then have an isomorphism 
\begin{equation}
\Psi_{\chi_{1}}:f\in\mathcal{H}(M,\rho_{M})\overset{\simeq}{\mapsto}f\chi_{1}\in\mathcal{H}(M,\rho_{M}\chi_{1}).
\end{equation}
This gives a $\Psi_{\chi_{1}}$-equivariant isomorphism
\[
(c\text{-ind}_{U_{M}(F)}^{M(F)}\psi_{M})^{\rho_{M}}\cong(c\text{-ind}_{U_{M}(F)}^{M(F)}\psi_{M})^{\rho_{M}\chi_{1}}.
\]
We thus have a $\Psi^{M}\circ\Psi_{\chi_{1}}^{-1}$-equivariant isomorphism
\begin{equation}
(c\text{-ind}_{U_{M}(F)}^{M(F)}\psi_{M})^{\rho_{M}\chi_{1}}\cong(c\text{-ind}_{U(F)}^{G(F)}\psi)^{\rho}.\label{eq:vec}
\end{equation}
By Equations (\ref{eq:IGrho}) and (\ref{eq:IGrho_sub}), it follows
that $\Psi^{M}$ restricts to an algebra isomorphism
\begin{equation}
\mathcal{H}_{W(G,T),\chi}\overset{\simeq}{\rightarrow}\mathcal{H}_{W(M,T),\chi}.
\end{equation}
 From the proof of \cite{Ro98}*{Theorem 4.15}, $W(M,T)_{\chi}=W(M,T)_{\chi\chi_{1}}$
and therefore $\Psi_{\chi_{1}}$ restricts to an isomorphism 
\[
\mathcal{H}_{W(M,T),\chi}\cong\mathcal{H}_{W(M,T),\chi\chi_{1}}.
\]
 Thus $\Psi^{M}\circ\Psi_{\chi_{1}}^{-1}$ restricts to an isomorphism
\begin{equation}
\mathcal{H}_{W(M,T),\chi\chi_{1}}\cong\mathcal{H}_{W(G,T),\chi}.\label{eq:subalg}
\end{equation}
 Note that if $G$ has connected center, then so does $M$ (see proof
of \cite{Car85}*{Propositon 8.1.4} for instance for this fact). Thus,
from Equations (\ref{eq:vec}) and (\ref{eq:subalg}), it follows
that to prove Theorem \ref{thm:main}, we can and do assume without
loss of generality that $\chi$ has depth-zero. 
\begin{rem}
For a much more general statement of the isomorphism $\Psi^{M}\circ\Psi_{\chi_{1}}^{-1}$,
see \cite{AdMi}*{\S 8}.
\end{rem}

\subsection{Proof in depth-zero case}

For results of this section, no restriction on characteristic or residue
characteristic is imposed. 

Let $I$ be the Iwahori
subgroup of $G$ which is in good position with respect to $(\bar{B},T)$
(note here that we are taking opposite Borel) and let $I_{0+}$ denote
its pro-unipotent radical. Put $T(F)_{0+}=I_{0+}\cap T(F)_{0}$. Then
$I/I_{0+}\cong T(F)_{0}/T(F)_{0+}$. Since $\chi$ is depth-zero,
it factors through $T(F)_{0}/T(F)_{0+}$ and consequently lifts to
a character of $I$ which we denote by $\rho$. The pair $(I,\rho)$
is then a $G$-cover of $(T(F)_{0},\chi)$ \cite{Hai}. 

Define $\phi:G(F)\rightarrow\mathbb{C}$ to be the function supported
on $U(F).(I\cap\bar{B})$ such that $\phi(ui)=\psi(u)\chi(i)$ for
$u\in U(F)$ and $i\in I\cap\bar{B}$. There is an isomorphism of
$G(\mathbb{F}_{q})$-spaces
\[
(c\text{-ind}_{U(F)}^{G(F)}\psi)^{G(F)_{x,0+}}\cong\text{ind}_{U(\mathbb{F}_{q})}^{G(\mathbb{F}_{q})}\text{\ensuremath{\underbar{\ensuremath{\psi}}}}.
\]
 Under this isomorphism, $\phi$ maps to a function $\text{\ensuremath{\underbar{\ensuremath{\phi}}}}:G(\mathbb{F}_{q})\rightarrow\mathbb{C}^{\times}$
which is supported on $U(\mathbb{F}_{q}).\bar{B}(\mathbb{F}_{q})$
and such that $\text{\ensuremath{\underbar{\ensuremath{\phi}}}}(ub)=\text{\ensuremath{\underbar{\ensuremath{\psi}}}}(u)\chi(b)$
for $u\in U(\mathbb{F}_{q})$ and $b\in\bar{B}(\mathbb{F}_{q})$.
Now $(\text{ind}_{U(\mathbb{F}_{q})}^{G(\mathbb{F}_{q})}\text{\ensuremath{\underbar{\ensuremath{\psi}}}})^{\chi}$
is an irreducible $\mathcal{H}_{W,\chi}$-module isomorphic to the
$\chi$-isotypical component of the irreducible $\text{\ensuremath{\underbar{\ensuremath{\psi}}}}$-generic
constituent of $\text{ind}_{\bar{B}(\mathbb{F}_{q})}^{G(\mathbb{\mathbb{F}}_{q})}\chi$.
Observe that $\text{\ensuremath{\underbar{\ensuremath{\phi}}}}\in(\text{ind}_{U(\mathbb{F}_{q})}^{G(\mathbb{F}_{q})}\text{\ensuremath{\underbar{\ensuremath{\psi}}}})^{\chi}$.
If $\chi$ is trivial then $(\text{ind}_{U(\mathbb{F}_{q})}^{G(\mathbb{F}_{q})}\text{\ensuremath{\underbar{\ensuremath{\psi}}}})^{\chi}$
corresponds to the Steinberg constituent of $\text{ind}_{\bar{B}(\mathbb{F}_{q})}^{G(\mathbb{\mathbb{F}}_{q})}\chi$.
If $G$ has connected center, then it is shown in \cite{Re02}*{\S 7.2, 2nd last paragraph}
that $(\text{ind}_{U(\mathbb{F}_{q})}^{G(\mathbb{F}_{q})}\text{\ensuremath{\underbar{\ensuremath{\psi}}}})^{\chi}\cong\mathrm{sgn}$.
Thus in either case, the $1$-dimensional space spanned by $\text{\ensuremath{\underbar{\ensuremath{\phi}}}}$
affords the $\mathrm{sgn}$ representation of $\mathcal{H}_{W,\chi}$.
Consequently, the $1$-dimensional space spanned by $\text{\ensuremath{\phi}}$
affords the $\mathrm{sgn}$ representation $\mathcal{H}_{W,\chi}$. 

It is readily checked that $\phi$ maps to $1$ under the isomorphism
of Theorem \ref{thm:cyc}. It follows that $\phi$ is a generator
of the cyclic $\mathcal{H}(G,\rho)$-module $(c\text{-ind}_{U(F)}^{G(F)}\psi)^{\rho}$.
We have by Frobenius reciprocity: 
\[
\mathrm{Hom}_{\mathcal{H}_{W,\chi}}(\mathrm{sgn},(c\text{-ind}_{U(F)}^{G(F)}\psi)^{\rho})\cong\mathrm{Hom}_{\mathcal{H}(G,\rho)}(\mathcal{H}(G,\rho)\underset{\mathcal{H}_{W,\chi}}{\otimes}\mathrm{sgn},(c\text{-ind}_{U(F)}^{G(F)}\psi)^{\rho}).
\]
This isomorphism sends $1\mapsto\phi$ to the element $1\otimes1\mapsto\phi$.
Theorem \ref{thm:main} now follows from the fact that $\mathcal{H}(G,\rho)\underset{\mathcal{H}_{W,\chi}}{\otimes}\mathrm{sgn}$
and $(c\text{-ind}_{U(F)}^{G(F)}\psi)^{\rho}$ are free $\mathcal{H}(T,\chi)$-modules
generated by $1\otimes1$ and $\phi$ respectively. 

\section{acknowledgments}

The first named author benefitted from discussions with Dipendra Prasad. 
The authors are thankful to the anonymous referee for pointing out a mistake in an earlier draft of this article.

\end{document}